\documentclass[12pt]{amsart}
\usepackage{amssymb,amsthm,enumerate}

\newcommand{\FF}{{\mathbb{F}}}

\newcommand{\ZZ}{{\mathbb{Z}}}
\newcommand{\NN}{{\mathbb{N}}}

\newcommand{\calO}{{\mathcal O}}

\newcommand{\bG}{{\mathbf{G}}}

\newcommand{\bT}{{\mathbf{T}}}
\newcommand{\bZ}{{\mathbf{Z}}}

\newcommand{\bC}{{\mathbf{C}}}

\newcommand{\Aut}{{\operatorname{Aut}}}

\newcommand{\Irr}{{\operatorname{Irr}}}
\newcommand{\GL}{{\operatorname{GL}}}

\newcommand{\SL}{{\operatorname{SL}}}
\newcommand{\PSL}{{\operatorname{PSL}}}

\newcommand{\GU}{{\operatorname{GU}}}
\newcommand{\Sp}{{\operatorname{Sp}}}

\newcommand{\Syl}{{\operatorname{Syl}}}

\newcommand{\tw}[1]{{}^#1\!}
\newcommand\nor{\triangleleft\,}

\newcommand{\diag}{{\rm diag}}
\newcommand{\St}{{\sf St}}

 \def\Q{\mathbb Q}
\def\irr#1{{\rm Irr}(#1)}

\def\ibr#1{{\rm IBr}(#1)}
\def\syl#1#2{{\rm Syl}_#1(#2)}
\def\norm#1#2{{\bf N}_{#1}(#2)}
\def\sbs{\subseteq}
\def\cent#1#2{{\bf C}_{#1}(#2)}
\def\zent#1{{\bf Z}(#1)}
\def\oh#1#2{{\bf O}_{#1}(#2)}

\newcommand{\bfN}{{\mathbf N}}
\newcommand{\bfZ}{{\mathbf Z}}

\newcommand{\AAA}{{\mathsf A}}
\newcommand{\SSS}{{\mathsf S}}

\renewcommand{\mod}{\bmod \,}
\newtheorem{thm}{Theorem}[section]
\newtheorem{lem}[thm]{Lemma}
\newtheorem{conj}[thm]{Conjecture}
\newtheorem{cor}[thm]{Corollary}

\newtheorem*{thmA}{Theorem A}
\newtheorem*{thmC}{Theorem C}
\newtheorem*{thmD}{Theorem D}

\newtheorem*{conjB}{Conjecture B}
\newtheorem*{thmE}{Theorem E}
\newtheorem*{thmF}{Theorem F}

\theoremstyle{definition}

\theoremstyle{remark}
\newtheorem{rem}[thm]{Remark}

\begin{document}

\title{Brauer Correspondent Blocks with One Simple Module}

\date{May 20, 2016}

\author{Gabriel Navarro}
\address{Departament of Mathematics, Universitat de Val\`encia,
 Dr. Moliner 50, 46100 Burjassot, Spain.}
\email{gabriel.navarro@uv.es}

\author{Pham Huu Tiep}
\address{Department of Mathematics, University of Arizona, Tucson, AZ 85721, USA}
\email{tiep@math.arizona.edu}

\author{Carolina Vallejo}
\address{ICMAT, Campus Cantoblanco UAM, C/ Nicol\'as Cabrera, 13-15, 28049 Madrid, Spain}
\email{carolina.vallejo@icmat.es}


\thanks{
The research of the first and third authors is partially supported by the
 Spanish Ministerio de Educaci\'on y Ciencia proyecto MTM2016-76196-P and Prometeo
 Generalitat Valenciana.  
 The second author
gratefully acknowledges the support of the NSF (grants DMS-1201374 and DMS-1665014). The third author acknowledges financial support from the Spanish Ministry of Economy and Competitiveness, through the ?Severo Ochoa Programme for Centres of Excellence in R\&D? (SEV-2015-0554).} 

\begin{abstract}
One of the main problems in representation theory
is to understand the exact relationship between
Brauer corresponding blocks of finite groups.
The case where the local correspondent has a
unique simple module seems key. We characterize this situation
for the principal $p$-blocks where $p$ is odd.
\end{abstract}

\maketitle

\vskip 2pc

%%\pagestyle{myheadings}
%%\markboth{for personal use only}{}
%%\markboth{}{}

%%%%%%%%%%%%%%%%%%%%%%%%%%%%%%%%%%%%%%%%%%%%%%%%%%%%%%%%%%%%%%%%%%%%%%%%%
\section{Introduction}
 
 Let $G$ be a finite group,  let $p$ be a prime, and let $\FF$
 be an algebraically closed field of characteristic $p$.
 The blocks of $G$ are the indecomposable two-sided ideals
 of the group algebra $\FF G$. Richard Brauer associated to
 each block $B$ of $G$
 a $p$-subgroup $D$ of $G$, up to conjugation, and a block $b$ of
 the local subgroup $\norm GD$, which is  called the  Brauer first main correspondent 
 of $B$. What is the exact relationship between these two algebras, and what invariants
 they share is one of the main problems in representation theory of finite groups.
 Our major interest is in the invariants $k(B)$, $k_0(B)$ and $l(B)$
 (which are the number of complex irreducible characters in $B$,
 those of them which have  height zero, and the number of simple modules in $B$ over $\FF$, respectively)
 and their relation with $k(b)$, $k_0(b)$ and $l(b)$.
 For instance, $k_0(B)=k_0(b)$ is the Alperin-McKay conjecture, and $l(B)\ge l(b)$ 
 would be a  consequence of the Alperin weight conjecture.
 
 \medskip
 
In this paper, we wish to understand how the local condition $l(b)=1$ affects $B$, and the other way around. 
 Already the key case where $B$ and $b$ are the principal blocks (the blocks containing the trivial representation of the group)
 is hard to handle. This is our main result.

\medskip
 
\begin{thmA} \label{thmA}
Let $G$ be a finite group, let $p$ be an odd prime, and let $P \in \syl pG$. 
Let $B$ be the principal block of $G$, and let $b$ be 
the principal block of $\norm GP$.
Then $B$ contains no non-trivial $p$-rational
height zero  irreducible character if and only if
 $l(b)=1$. 
 \end{thmA}

\medskip
As we will point out in Conjecture \ref{Conj2} below, we have an ad hoc statement for the prime $p=2$, but to 
prove it seems presently out of reach. As it seems also out of reach to prove the following.

\medskip
\begin{conjB}
Let $G$ be a finite group, let $p$ be an odd prime, let $B$ be 
a $p$-block of $G$ and let 
$b$ be the Brauer first main correspondent of $B$.
If $B$ contains exactly one $p$-rational
height zero  irreducible character, then $l(b)=1$.
\end{conjB}

\medskip
Outside principal blocks, the converse of Conjecture B is not true, even in blocks
with normal maximal defect. For instance, the {\tt SmallGroup}(72,22) in \cite{GAP} is a counterexample
for $p=3$.

\medskip

There is a related   characterization
of when $l(b)=1$  for $p$-solvable groups.
If $\chi$ is an ordinary character of $G$, then $\chi^0$ is the Brauer character obtained
by restricting $\chi$ to the $p$-regular elements of $G$.

\medskip

\begin{thmC}
Suppose that $G$ is $p$-solvable, with $p$ odd. Let $B$ be a $p$-block
with defect group $P$ and let $b$ be its Brauer first main correspondent.
Then $l(b)=1$ if and only if there is exactly one $p$-rational $\chi \in \irr B$
 of  height zero and such that $\chi^0 \in \ibr B$.
\end{thmC}

\medskip
Unfortunately, 
the ``only if" direction of Theorem C is false outside $p$-solvable groups.
Let $G=\AAA_6$, $p=3$ and $B$ the principal block of $G$.
Then $B$ contains a unique 
$p$-rational $p'$-degree irreducible
character that lifts an irreducible
Brauer character while $l(b)=4$.
In this case, the defect groups
of $B$ are abelian. The same
situation
happens in ${\rm PSL}_2(p)$ when $p \geq 5$. 
In this case, $l(b) = (p-1)/2$ and the defect groups of
$B$ are   cyclic.
It is interesting to speculate to what extent
the ``only if" direction holds. 
%, and seems quite plausible at least for blocks with abelian defect groups.)
\medskip

For a character theorist it is always pleasant to find
new properties of a finite group which can be read off from its
character table.  By a result of R. Brauer,
 the principal block of a group 
has a unique irreducible Brauer character if and only if it has a normal $p$-complement
(see Corollary 6.13 of \cite{Nav98}). Hence Theorem A is equivalent to  the following.

\begin{thmD}
Let $G$ be a finite group, let $p$ odd, and let $P \in \syl pG$. 
Then $\norm GP$ has a normal $p$-complement if  and only if there are no non-trivial $p$-rational $p'$-degree complex irreducible
characters in the principal block of $G$.
\end{thmD}
 
 In general, it is not easy to produce $p$-rational
 irreducible characters. Even with the strong hypotheses that  $\theta \in \irr N$ 
 is a $p$-rational character of $p'$-degree
 in the principal block of $N\nor G$, $G/N$
 is cyclic of $p'$-order and $\theta$
 extends to $G$, then
 it is not necessarily true  that $\theta$ has a $p$-rational extension to $G$.
 Our way to produce $p$-rational characters
 is indirect, by using some  results  which we believe are of independent
 interest. The first of these  is a relative  to normal subgroups version of the Glauberman correspondence.

 If $P$ is a group acting by automorphisms on $G$,
 then ${\rm Irr}_P(G)$ is the set of $P$-invariant irreducible
 characters of $G$. In Theorem E below, the Glauberman correspondence
 is obtained when $N=1$. If $\chi$ is a character, then we denote by
 $\Q(\chi)$ the smallest field containing the values of $\chi$.
   
 \medskip

\begin{thmE}\label{thmE}
Suppose that a $p$-group $P$ acts as automorphisms
on a
finite group $G$. Let $N \nor G$ be $P$-invariant such that $G/N$ is a $p'$-group.
Let $C/N=\cent{G/N}P$. Then there exists a natural bijection
$^*:{\rm Irr}_P(G) \rightarrow {\rm Irr}_P(C)$. In fact, if $\chi \in {\rm Irr}_P(G)$, then
$$\chi_C=e\chi^* + p \Delta + \Xi \, ,$$
where $\Delta$ and $\Xi$ are characters of $C$ or zero, $p$ does not divide $e$, and no irreducible
constituent of $\Xi$ lies over some $P$-invariant character of $N$. In fact, $e \equiv \pm 1 $ mod $p$.
In particular, $\Q(\chi)=\Q(\chi^*)$.
Also, if  $\chi$ has $p'$-degree, then $\chi$ lies in the principal block of $G$ if and only if $\chi^*$ lies in the principal block of $C$.
 \end{thmE}

\medskip
Several particular cases of Theorem E
have appeared previously in the literature (see, for instance, Theorem 5.1
of \cite{IN12}).
\medskip

We will also need a result on extension of characters
that generalizes  results of Alperin and Dade (see
\cite{A} and \cite{D}).

\begin{thmF}Suppose that $N \nor G$.
Let $\theta \in \irr N$ be $p$-rational, $G$-invariant of
$p'$-degree in the principal block
of $N$, where $p$ is odd. Let $Q \in \syl pN$,
and assume that $|G:N\cent GQ|$ is a $p$-power.
Then $\theta$ uniquely determines a character $\chi \in \irr G$ in the principal
block of $G$ such that
$\chi$ is $p$-rational and $\chi_N=\theta$.
\end{thmF}

Under the hypotheses
of Theorem F, it is false that $\theta$ has a unique $p$-rational extension in the principal block
of $G$. For instance, take $p=3$,  $G=C_3 \times \SSS_3$, $N=C_3$,
and $\theta$ the principal character of $N$. In this case, the character $\chi$
determined by Theorem F
  is the trivial character of $G$, but there is another $p$-rational extension
of $\theta$ to $G$.

%%%%%%%%%%%%%%%%%%%%%%%%%%%%%%%%%%%%%%%%%%%%%%%%%%%%%%%%%%%%%%%%%%%%%%%%%
\section{The Relative Glauberman Correspondence}

\noindent We follow the notation of  \cite{Isa76} for ordinary characters
and the notation of
\cite{Nav98} for   modular characters and blocks. 
In particular, if $p$ is a prime number, and ${\bf R}$ is the ring
of algebraic integers in $\mathbb C$, we choose $M$ a maximal ideal
of $\bf R$ containing $p{\bf R}$, with respect to which the Brauer
characters of any finite group $G$ are constructed. We also
let $^*: {\bf R} \rightarrow {\bf R}/M$ be the canonical
ring epimorphism. (Later on, we will also denote by $^*$ several character correspondences,
but we believe that there is no risk of confusion.)
If $N \nor G$ and $\theta \in \irr N$, then $\irr{G|\theta}$
is the set of irreducible constituents of the induced character $\theta^G$.
Also, $G_\theta$ is the stabilizer of $\theta$ in $G$.
Sometimes, we will denote by $B_0(G)$ the set of the irreducible
complex characters of $G$ which lie in the principal $p$-block of $G$,
where $p$ is a prime. By a block, we mean a $p$-block.
Throughout this paper, we will
denote by ${\mathcal G}={\rm Gal}(\bar \Q/\Q)$ the absolute Galois
group. By elementary character theory, we know that
$\mathcal G$ acts on the irreducible complex characters
of every finite group $G$. 
\medskip

We begin by
proving the following.

\begin{lem}\label{Pinvariantover} Suppose that $P$ is a $p$-group acting 
as automorphisms on a finite group $G$.
Suppose that $N \nor G$ is $P$-invariant with $G/N$ a $p'$-group.
Let $\theta \in \irr N$ be $P$-invariant. If $P$ acts trivially on $G/N$,
then every $\psi \in \irr{G|\theta}$ is $P$-invariant.
\end{lem}
\begin{proof}
By the Clifford correspondence (Theorem 6.11 of \cite{Isa76}),
we may assume that $\theta$ is $G$-invariant.
Let $g \in G$, $x \in P$ and $\psi \in \irr{G|\theta}$.
We want to show that $\psi^x(g)=\psi(g)$. Write $H=N\langle g \rangle$,
a $P$-invariant subgroup of $G$. By considering the irreducible
constituents of $\chi_H$, all of which lie over $\theta$, we may assume that
$H=G$. That is to say, we assume that $G/N$ is cyclic. 
Hence $\theta$ extends to $G$, by Corollary 11.22 of \cite{Isa76}.
By coprime action (Theorem 13.31 of \cite{Isa76}), 
there is $\chi \in \irr{G|\theta}$ which is $P$-invariant.
By Gallagher's theorem (Corollary 6.17 of \cite{Isa76}), every $\psi\in \irr{G|\theta}$ is of the form $\beta \chi$ for $\beta \in \irr{G/N}$.
Since $P$ acts trivially on $G/N$, then
 every $\beta \in \irr{G/N}$ is $P$-invariant, and the statement follows.
\end{proof}

\begin{lem}\label{GlaubermanLemma}
Suppose that $P$ is a $p$-group acting 
as automorphisms on a finite group $G$.
Suppose that $N \nor G$ is $P$-invariant with $G/N$ a $p'$-group, and let $C/N=\cent {G/N} P$. Suppose that $\chi \in \irr G$ is $P$-invariant. 
Then $\chi_N$ has a $P$-invariant constituent $\theta \in \irr N$, and any two such constituents are $C$-conjugate.
\end{lem}

\begin{proof}
The first part is Theorem 13.27 of \cite{Isa76}. The second part
follows from Corollary 13.9 of \cite{Isa76}.
\end{proof}

We can now prove Theorem E.

\begin{thm}[Relative Glauberman Correspondence]\label{relativeGlauberman} Suppose that $P$ is a $p$-group acting 
as automorphisms on a finite group $G$. Let $N \nor G$ be $P$-invariant such that $G/N$ is a $p'$-group.
Let $C/N=\cent{G/N}P$. Then there exists a natural bijection
$$^*:{\rm Irr}_P(G) \rightarrow {\rm Irr}_P(C).$$ 
In fact,
$$\chi_C=e\chi^* + p \Delta + \Xi \, ,$$
where $\Delta$ and $\Xi$ are characters of $C$ or zero, $e \equiv \pm 1 $ mod $p$, and no irreducible
constituent of $\Xi$ lies over some $P$-invariant character of $N$.
In particular, $$\Q(\chi)=\Q(\chi^*).$$
Also, if
$\chi$ has $p'$-degree, then $\chi$ lies in the principal block of $G$ if and only if $\chi^*$ lies in the principal block of $C$.\end{thm}

\begin{proof} We first prove the   part of the statement concerning the existence of a  bijection. 

Notice that $C$ acts on ${\rm Irr}_P(N)$. Indeed, 
if $\theta \in {\rm Irr}_P(N)$,
$x \in P$ and $c \in C$, then $c^x=nc$ for some $n \in N$.
Hence $(\theta^c)^x =\theta^{x^{-1}cx}=\theta^c$,
and $\theta^c$ is $P$-invariant.
Let $\Lambda$ be a complete set of representatives of the $C$-orbits on ${\rm Irr}_P(N)$.
We claim that
$${\rm Irr}_P(G)= \bigcup_{\theta \in {\Lambda}} {\rm Irr}_P(G|\theta) \,  $$
is a disjoint union.
Let $\chi \in {\rm Irr}_P(G)$. By Lemma \ref{GlaubermanLemma}
we have that $\chi_N$ has a $P$-invariant irreducible constituent $\theta$, and that all of them are $C$-conjugate.
This proves the claim.
By the same argument, we have that 
$${\rm Irr}_P(C)= \bigcup_{\theta \in {\Lambda}} {\rm Irr}_P(C|\theta) \, $$
is a disjoint union. Then
  it suffices to prove that there are bijections
$$^*:{\rm Irr}_P(G|\theta) \rightarrow {\rm Irr}_P(C|\theta)$$
satisfying the conditions in the statement of the theorem.
We prove this  by induction on $|G:N|$.

\smallskip

Let $\chi \in {\rm Irr}_P(G)$, let $\theta \in {\Lambda}$
be under $\chi$, let $T$ be the stabilizer of $\theta$ in $G$, and let $\psi
 \in \irr{T|\theta}$ be the Clifford correspondent of $\chi$.
Let $\mathbb T$ be a set of representatives of the double cosets of $T$ and $C$ in $G$ with $1 \in \mathbb T$. By the
Mackey formula, we have that
$$\chi_C= (\psi_{T\cap C})^C + \delta \, ,$$
where $\delta =\sum_{1 \neq t \in \mathbb T}(\psi^t_{T^t\cap C})^C$. We claim that no irreducible constituent of $\delta$ lies over $\theta$. Otherwise, let  $\eta$ be an irreducible constituent of $(\psi^t_{T^t\cap C})^C$ for some $1\neq t \in \mathbb T$ lying over $\theta$. Then $\eta$ lies over $\theta$ and over $\theta^t$. By Clifford's theorem (Theorem 6.2 of \cite{Isa76}), we have that
 $\theta=\theta^{tc}$ for some $c \in C$, but this is a contradiction since $1\neq t\in \mathbb T$. This proves the claim.
 
 Notice that, in fact, no irreducible constituent of $\delta$ lies over any $P$-invariant irreducible character $\tau \in \irr N$.
Otherwise,  $\tau$ and $\theta$
are  $P$-invariant characters of $N$ lying under $\chi$. By Lemma \ref{GlaubermanLemma} $\tau$ is $C$-conjugate to $\theta$, and thus
$\theta$ lies under $\delta$, a contradiction. 
 
 Suppose now that $T<G$. By induction,
 there is a bijection $$^*:{\rm Irr}_P(T|\theta) \rightarrow {\rm Irr}_P(T\cap C|\theta)$$
 such that $\psi_{T\cap C}=e\psi^* + p \Delta$,
 where $\Delta$ is a character or zero, and $e \equiv \pm 1 (\mod p)$. 
 Then 
 $$\chi_C=(\psi_{T\cap C})^C + \delta =e(\psi^*)^C + p \Delta^C + \delta \, .
 $$
By   the Clifford correspondence (Theorem 6.11 of \cite{Isa76}), we know
that induction defines bijections
 ${\rm Irr}_P(T|\theta) \rightarrow {\rm Irr}_P(G|\theta)$
 and  
 ${\rm Irr}_P(T\cap C|\theta) \rightarrow {\rm Irr}_P(C|\theta)$.
 Since 
 $$\chi^*=(\psi^*)^C \in \irr{C|\theta},$$
 we conclude that
 we may assume that $\theta$ is $G$-invariant.

 \smallskip

We have to show that  
for $\chi \in {\rm Irr}_P(G|\theta)$, we have that
that $\chi_C=e\chi^* + p\Delta$, where $\chi^* \in \irr C$,
$p$ does not divide $e$,
and that the map $\chi \mapsto \chi^*$ is a bijection
${\rm Irr}_P(G|\theta) \rightarrow {\rm Irr}_P(C|\theta)$.
We consider the semidirect product $\Gamma=GP$ of $G$ by $P$. Since $\theta$ is $\Gamma$-invariant, we have that $(\Gamma, N, \theta)$ is a character triple.
By Theorem 11.28 of \cite{Isa76} there is an isomorphism $(\tau, \sigma)\colon (\Gamma, N, \theta) \rightarrow (\Gamma^\tau, N^\tau, \theta^\tau)$  of character triples, where $N^\tau$
is central in $\Gamma^\tau$. 
Recall that $\tau: \Gamma/N \rightarrow \Gamma^\tau/N^\tau$
is a group isomorphism.
We are going to write
$\tau(H/N)=H^\tau/N^\tau$
for every subgroup $N \leq H \leq \Gamma$. Since $(NP)^\tau/N^\tau$ is
a $p$-subgroup of $\Gamma^\tau/N^\tau$, and $N^\tau$ is central, then $(NP)^\tau$ has a unique Sylow $p$-subgroup which we denote by 
$P^\tau$. Now $P^\tau$ acts on $G^\tau/N^\tau$ the same way as $(PN)^\tau$ acts on $G^\tau/N^\tau$. Hence, by the properties of
character triple isomorphisms in Definition 11.23 of
\cite{Isa76}, it is  
 no loss to  assume that $N \leq \zent \Gamma$. 
 Hence we may assume that
 $[N,P]=1$ and that $N \leq \zent G$.
 In particular, $G$ has a central Sylow $p$-subgroup
 $N_p$, a normal $p$-complement $K$, and
 in particular $C=\cent K P \times N_p$. 
 Write $\theta=\theta_{p'} \times \theta_p$,
 where $\theta_{p'}=\theta_{K\cap N}$
 and $\theta_p=\theta_{N_p}$.
 We have that
 $${\rm Irr}_P(G|\theta)=\{ \mu \times \theta_p \, |\, \mu \in 
 {\rm Irr}_P(K|\theta_{p'}) \}$$ 
 and
 $${\rm Irr}_P(C|\theta)={\rm Irr}(C|\theta)=\{ \epsilon \times \theta_p
  \, |\, \epsilon \in \irr{C\cap K|\theta_{p'}} \}.$$
  By  Theorem 13.29 of \cite{Isa76},
  we have that
  the Glauberman correspondence 
$$^*: {\rm Irr}_P(K) \rightarrow \irr{C\cap K}$$
  sends ${\rm Irr}_P(K|\theta_{p'})$ bijectively onto 
   $\irr{C\cap K|\theta_{p'}}$. 
   Since   $\mu_{C\cap K}=e\mu^*  + p \Delta$,
   where $e \equiv \pm 1 (\mod p)$, the first part of the proof of the
   statement
   is now complete.
 
\medskip
The action of the absolute Galois group ${\mathcal G}$ on characters
commutes with the action of $P$ and with restriction
of characters. Hence our map
$^*:{\rm Irr}_P(G) \rightarrow {\rm Irr}_P(C)$ is
$\mathcal G$-equivariant. This  implies that
$$\Q(\chi)=\Q(\chi^*)$$
for $\chi \in {\rm Irr}_P(G)$.

\smallskip

We finally prove the statement about blocks.  Let $\chi \in \irr G$ be $P$-invariant of $p'$-degree. We have that $\chi_C=e\chi^*+p\Delta+\Xi$, where $p$ does not divide $e$ and no irreducible constituent of $\Xi$ 
lies over a $P$-invariant character of $N$. 
We prove that $\chi$ lies in the principal block
of $G$ if and only if $\chi^*$
lies in the principal block of $C$. We proceed
 by induction on $|G:N|$. 

Let $\theta \in \irr N$ be $P$-invariant under $\chi$. Let $T$
be the stabilizer of $\theta$
in $G$, and let $\psi \in \irr{T|\theta}$ be the Clifford correspondent of $\chi$. We have that  $\psi(1)$ is a $p'$-number and $\psi$ is $P$-invariant. By the first part of the proof $\psi_{T\cap C}=f\psi^*+p\Delta'$, where $p$ does not divide $f$,
and we know that   $\psi^* \in \irr{T\cap C}$ is the Clifford correspondent of $\chi^*$. By induction,
 if $T<G$, then $\psi  \in B_0(T)$ if and only if $\psi^* \in B_0(T\cap C)$. Thus, in this case the statement follows from  
 Corollaries 6.2 and  6.7 of \cite{Nav98}.
 \smallskip

We may assume that $\theta$ is $G$-invariant, 
and therefore we have that
$$\chi_C=e\chi^*+p\Delta$$ and so $\chi^*$ has $p'$-degree. Again, let $\Gamma=GP$ be the semidirect product of $G$ and $P$. Since $NP$ has $p'$-index in $\Gamma$, we can choose $P \leq R $ a Sylow
$p$-subgroup of $\Gamma$ contained in $NP$, so that $NP=NR$. 
Also $\norm {\Gamma/N}{NP/N}=\norm {\Gamma/N}{NR/N}$,
and we see that
$\cent{G/N}P=\cent{G/N}R$ and that ${\rm Irr}_P(G)={\rm Irr}_R(G)$.
 Write 
 $$M/N=\norm {\Gamma/N}{NR/N}=N \norm \Gamma R/N,$$ 
so that  $M \cap G=C$.
By Corollary 9.6 of \cite{Nav98}, let $B$ be the unique block of $\Gamma$ covering the block of $\chi$ and let $b$ be the unique block
of $M$ covering the block of $\chi^*$. 
Since $\chi$ has $p'$-degree, it enters with $p'$-multiplicity in $(1_1)^G=((1_P)^\Gamma)_G$. If $\psi \in \irr{\Gamma}$ lies over
$\chi$, then $[\psi_G, \chi]$ is a $p$-power, by Corollary 11.29 of \cite{Isa76} and
and using that $\chi$ is $\Gamma$-invariant. 
Therefore $\chi$ extends to
some $\tilde\chi \in \irr{\Gamma}$.
By the same argument $\chi^*$ extends
to $\tilde\chi^* \in \irr{M}$. Of course, $B=B_0(\Gamma)$ if and only if  $\chi$ belongs to the principal block of $G$ and $b=B_0(M)$ if  and only if 
$\chi^*$ belongs to the principal block of $C$ (using Corollary 9.6 of \cite{Nav98}).
%We may assume $\tilde{\chi^*}$ lies under $\tilde \chi$, i.e., some extension of $\chi^*$ lies under $\tilde \chi$ by Gallagher's theorem.
% Otherwise we would have $$\tilde \chi_Y=\sum_{\beta \in \irr {Y/C}} e_\beta \beta \tilde{\chi*}+ \rho,$$ where no irreducible constituent of $\rho$ lies over $\chi^*$. Since  $\tilde \chi_C= \chi_C=e \chi^*+p\Delta +\Xi$ where $p$ does not divide $e$, we conclude that $p$ does not divide some $e_\beta \beta(1)$. 

Since $\tilde\chi$ and
$\tilde\chi^*$ have $p'$-degree,
then we know that $B$ and $b$ have defect group $R$, by Theorem 4.6
of \cite{Nav98}.
By Problem 4.5 of \cite{Nav98}, we have that $B=B_0(\Gamma)$ if  and only if
$$\left (\frac{|{\rm cl}_\Gamma(x)|\tilde \chi(x)}{\tilde\chi(1)} \right )^*=|{\rm cl}_\Gamma(x)|^*$$ for every
$p$-regular $x \in \Gamma$
such that $R \in \syl p{\cent \Gamma x}$.
 Similarly, $b=B_0(M)$ if and only if  
 $$\left (\frac{|{\rm cl}_M(y)|\tilde \chi^*(y)}{\tilde\chi^*(1)} \right)^*=|{\rm cl}_M(y)|^*$$ for 
 every $p$-regular
  $y \in M$ such that $R \in \syl p{\cent M y}$

\smallskip

Suppose that $K={\rm cl}_\Gamma(x)$, where
$x$ is $p$-regular and $R \in \syl p{\cent \Gamma x}$. 
Notice that $x \in G$, since $\Gamma/G$ is
a $p$-group. Now 
$$\cent GR N/N  \leq \cent {G/N}R=\cent{G/N}P=C/N,$$ 
and therefore $x \in \cent GR \leq C$.
Let $L={\rm cl}_M(x)$.
By Lemma 4.16 of \cite{Nav98} we have that $K\cap \cent \Gamma R$ is the
 conjugacy class of $x$ in  $\norm \Gamma R$. Also, 
 $$|K| \equiv |K\cap \cent \Gamma R| (\mod p)$$
 by counting. By the same argument,
 $L\cap \cent \Gamma R$ is the  conjugacy class of  $x$
 in $\norm \Gamma R$ and also
  $$|L| \equiv |L\cap \cent \Gamma R| (\mod p).$$
  Since $K\cap \cent \Gamma R=L \cap \cent \Gamma R$,
  we see that $|K| \equiv |L| (\mod p)$.
    Also, since
$$\chi_C=e\chi^*+p\Delta,$$
    we have that
    $\chi(x) \equiv e\chi^*(x) (\mod p)$ and $\chi(1) \equiv e\chi^*(1) (\mod p)$, where $p$ does not divide $e$. 
    Thus $\chi^*(1)\chi(x)\equiv \chi(1)\chi^*(x) (\mod p)$ and
    $\tilde\chi^*(1) \tilde\chi(x) \equiv \tilde\chi(1) \tilde\chi^*(x) (\mod p)$. 
    Since $|K|=|L| (\mod p)$,
    we deduce that
   $$|K|\tilde\chi^*(1) \tilde\chi(x) \equiv |L| \tilde\chi(1) \tilde\chi^*(x) \mod p \, .$$  
Using the fact that the degrees of $\chi$ and $\chi^*$ are $p'$-numbers,
    we deduce that
$$\left ( \frac{|K|\tilde\chi(x)}{\tilde\chi(1)} \right )^*=\left ( \frac{|L|\tilde\chi^*(x)}{\tilde\chi^*(1)}\right )^*.$$
The result follows from the discussion in the preceding paragraph
using, as we have proved, that $|{\rm cl}_\Gamma(x)|^*=|{\rm cl}_M(x)|^*$
for every $p$-regular $x \in \cent \Gamma R$.
 \end{proof}
 
\section{An extension theorem}
The aim of this section is to prove Theorem F. We first need some lemmas.  

\begin{lem}\label{p'principal} Suppose that $N \nor G$ and that
 $\psi \in \irr G$ has $p'$-degree and is such that $\psi_N=\theta \in \irr N$.
 Assume that $\psi_H$ belongs to the principal block of $H$
 whenever $H/N$ is a cyclic $p'$-group. Then $\psi$ belongs to the principal block of $G$.\end{lem}
\begin{proof} 
Since $\psi$ lies in a block
of maximal defect, by Problem 4.5 of \cite{Nav98},
 we want to show that
 $$\left( {|K|\psi(x) \over \psi(1)} \right)^*=|K|^*  ,$$
 where $K={\rm cl}_G(x)$ is the conjugacy class of
 a $p$-regular $x \in G$ 
 with $|G:\cent Gx|$  a $p'$-number. 
Since $\psi(1)$ is not divisible by $p$,
 it suffices to show that $\psi(x)^*=\psi(1)^*$.
Let $H=N\langle x\rangle$. 
We know that there is 
$P \in \syl pG$ such that $P \leq \cent Gx$.
Let $Q=P\cap N \in \syl p N$, so that 
 $Q  \leq \cent N x \leq \cent Hx$.
Since $H/N$ is a $p'$-number,
it follows that $Q \in \syl pH$. 
In particular, $p$ does not divide $|L|$, where $L={\rm cl}_H(x)$.
By hypothesis  $\psi_H$ belongs to the principal block
of $H$, and we conclude that
$$|L|^* \psi(x)^*=|L|^*\psi(1)^*.$$ 
Then
$\psi(x)^*=\psi(1)^*$, as desired. \end{proof}
\medskip
We remind the reader that, in general,
 if $\psi \in \irr G$ lies in the principal
block of $G$ and $\psi_H \in \irr H$, then $\psi_H$ needs not
to be in the principal block of $H$. For instance, take $G=\AAA_4$,  
$p=2$, and $H$ is a Sylow 3-subgroup of $G$. However, the following 
statement holds.

\begin{lem}\label{sub}
Suppose that $\psi \in \irr G$ lies in the principal block
of $G$, and assume that $H \nor\nor G$. 
If $\psi_H \in \irr H$, then $\psi_H$ lies in the principal block of
$H$.
\end{lem}

\begin{proof}
Arguing by induction on $|G:H|$,  we may assume that $H \nor G$.
Then the result  follows by Theorem 9.2 of \cite{Nav98}.
\end{proof}

\begin{lem}\label{pquotient}
Let $K \nor G$ with $G/K$ being a $p$-group and $p > 2$. If $\gamma \in \irr K$ is $p$-rational and $G$-invariant,
then $\gamma^G$ contains a unique $p$-rational irreducible
constituent $\hat \gamma \in \irr G$.
Furthermore, $\gamma$ lies in the principal block of $K$ if and only
if $\hat\gamma$ lies in the principal block of $G$.
\end{lem}
\begin{proof}
This is Theorem 6.30 of \cite{Isa76}
together with Corollary 9.6 of \cite{Nav98}.
\end{proof}

We can now prove Theorem F, which is a variation of Theorem 3.2 of \cite{NT11}. 

\begin{thm}\label{teoE} Suppose that $N \nor G$.
Let $\theta \in \irr N$ be $p$-rational, $G$-invariant of
$p'$-degree in the principal block
of $N$, where $p$ is odd. Let $Q \in \syl pN$,
and assume that $|G:N\cent GQ|$ is a power of $p$.
Then $\theta$ uniquely determines a character $\chi \in \irr G$ in the principal
block of $G$ such that
$\chi$ is $p$-rational and $\chi_N=\theta$.\end{thm}

\begin{proof}   
Let $M=N\cent GQ$. By the Frattini argument,
we have that $M \nor G$. 

We next show that if $N \leq U \leq M$ and $U/N$ has a normal $p$-complement,
 then there exists a unique 
$p$-rational extension $\eta_{(U)} \in \irr U$
of $\theta$ in the principal block of $U$. 
Let $V/N$ be the normal $p$-complement of $U/N$.
We have that $V=N\cent V Q$ and $V/N$ is a $p'$-group.
Since $V/N$ is a $p'$-group, then $\cent VQ/\cent NQ$
is a $p'$-group. By elementary group theory, $\zent Q$
is a central Sylow $p$-subgroup of $\cent VQ$, and therefore
there exists $Y \leq \cent VQ$ of $p'$-order such that $\cent VQ=Y \times \zent Q$.
By Theorem 3.2 of \cite{NT11}, there exists a unique $\hat\theta \in \irr V$
in the principal block of $V$ lying over $\theta$. In fact $\hat\theta_N=\theta$.
By uniqueness, $\hat\theta$ is $p$-rational and $U$-invariant.
(This is a standard argument. For instance,
if $\sigma \in \mathcal G$
fixes $\Q(\theta)$, then $\hat\theta^\sigma$ is
a $p$-rational extension of $\theta$ in the principal block,
so by uniqueness $\hat\theta^\sigma=\hat\theta$. Thus $\Q(\theta)=\Q(\hat\theta)$
and $\hat\theta$ is $p$-rational.)
By Lemma \ref{pquotient}, $\hat\theta$ has a unique $p$-rational extension $\eta$ to
$U$, which lies in the principal block of $U$.
If $\eta'$ is another $p$-rational extension of $\theta$ in the principal
block of $U$, then $\eta'_V= \rho \in \irr V$ lies in the principal block of $V$ (by Lemma \ref{sub}),
and extends $\theta$. By Theorem 3.2 of \cite{NT11},
$\rho=\hat\theta$. So $\eta'$ is a $p$-rational extension of $\hat \theta$,
and then $\eta'=\eta$ by Lemma \ref{pquotient}.

\smallskip

We now define a class function $\eta$ of $M$,
which is uniquely determined by $\theta$, as follows:
for $m \in M$, let $H=N\langle m \rangle  \leq M$,
and, by the previous paragraph, let $\eta_{(H)} \in \irr H$ be the unique
$p$-rational extension of $\theta$ in the principal block of $H$.
Set $\eta(m)=\eta_{(H)}(m)$.
It is straightforward to check
 that $\eta$ is a $G$-invariant
 class function of $M$ by using that $\theta$ is $G$-invariant and that
 $\eta_{(H^z)}=(\eta_{(H)})^z$ for $z \in G$. 
 Notice that $\eta(n)=\theta(n)$ for $n \in N$.
 
 Next we prove that $\eta$
is a generalized character. %by using Brauer's theorem 8.?? of \cite{Isa76}.
Suppose that $E/N$ is nilpotent, where $N \leq E \leq M$. By the second 
paragraph of this proof, there exists a unique $p$-rational $\psi \in \irr E$
 in the principal block extending $\theta$. We prove that
$\eta_E=\psi$. Let $g \in E$ and write $H=N\langle g\rangle$. Then $\psi_H$ is $p$-rational.
Since $H \nor \nor E$,
we have that $\psi_H$ lies in the principal block
of $H$ by Lemma \ref{sub}. Since $\psi_H$ extends $\theta$, then $\psi_H=\eta_{(H)}$. Consequently $\psi(g)=\eta(g)$, and $\psi_E=\eta_E$, as wanted.
By  Theorem 8.4(a) of \cite{Isa76}, we have that $\eta$ is a generalized character
of $M$.
By using Lemma 8.14(c) of \cite{Isa76} it is easy to prove that
$[\eta, \eta]=1$, so that $\eta \in \irr M$ by Theorem 8.12 of \cite{Isa76}.
Also, $\eta_N=\theta$. By Lemma \ref{p'principal}, we have that
$\eta$ lies in the principal block  of $M$ (because we have shown that if $E/N$
is nilpotent and $N \leq E \leq M$, then $\eta_E$ is the unique $p$-rational extension
of $\theta$ in the principal block of $E$).
Also $\eta$ is $p$-rational by definition.
We already know that $\eta$ is $G$-invariant.
By Lemma \ref{pquotient}, we know that there is a unique
$p$-rational $\chi \in \irr G$ extending $\eta$, which lies in the principal
block of $G$. \end{proof}

\medskip

The following result is a suitable extension of Theorem 6.1 of \cite{NTT07}.

\begin{cor}\label{thecor}
 Let $N \nor G$. Let $p$ be an odd prime and let $P\in \syl p G$. Suppose that $PN/N$ is self-normalizing in $G/N$.
Suppose that $\nu \in \irr N$ is $P$-invariant,
$p$-rational, has $p'$-degree,
and lies in the principal block of $N$. 
Then there exists a $p$-rational
$\chi \in \irr {G|\nu}$  of $p'$-degree
lying in the principal block of $G$.\end{cor}

\begin{proof} We proceed by induction on $|G:N|$. 

\smallskip

We may assume that $\nu$ is $G$-invariant. Indeed, let $T=G_\nu$
be the stabilizer of $\nu$ in $G$. 
If $T<G$ then, by the inductive hypothesis, there is a  $p'$-degree $p$-rational $\psi \in \irr T$ lying over $\nu$, in the principal block of $T$. Then, $\chi=\psi^G \in \irr {G|\nu}$ is $p$-rational and has $p'$-degree (for $PN\leq T$). Also,
by Corollary 6.2 and Theorem 6.7  of \cite{Nav98} $\chi$ lies in the principal block, as wanted.

\smallskip

Let $M/N$ be a chief factor of $G$. We claim that we
may assume that $G=MP$.
Notice  that $PM/N$ has a self-normalizing Sylow $p$-subgroup. If $MP<G$, then by 
the inductive hypothesis there is $\eta \in \irr{MP}$ of
$p'$-degree, $p$-rational lying over $\nu$, in the principal block
of $MP$.
Let $\tau=\eta_M \in \irr M$,
which is $p$-rational of $p'$-degree, $P$-invariant, in the principal block
of $M$. Since $PM/M$ is self-normalizing in $G/M$, again by the inductive hypothesis, there is a 
$p$-rational $\chi \in \irr G$  of $p'$-degree  lying over $\tau$
and in the principal block of $G$. 
Hence the claim follows.

\smallskip

Let $Q$ be a Sylow
$p$-subgroup of $N$. By the Frattini argument $G=N\norm G Q$. Then $N \cent MQ$ is normal in $G$ and so, either $M=N \cent M Q$ or $\cent M Q \leq N$. In the first case, the result follows from Theorem \ref{teoE} since $G/M$ is a $p$-group.

\smallskip

Assume finally that $\cent MQ \leq N$. 
In this case, by Lemma 3.1 of \cite{NT11}, the only block of $M$
covering the principal block of $N$
 is the  principal block of $M$, and the only block of $G$ covering 
 the principal block of $M$ is the principal block of $G$ because $G/M$
is a $p$-group (by Corollary 9.6 of \cite{Nav98}). Hence
the principal block of $G$ is the only block of $G$ covering the principal
block of $N$. By Theorem 6.1 of \cite{NTT07}, there exists $\chi \in \irr G$
of $p'$-degree, $p$-rational lying over $\nu$. Since $\nu$ lies in the principal
block of $N$, necessarily $\chi$ lies in the principal block of $G$ by Theorem 9.2
of \cite{Nav98}, and the proof of the statement is complete.
 \end{proof}

As we have said before, there are
examples where $G/N$
is a cyclic $p'$-group,
$\theta \in \irr N$ is $p$-rational of $p'$-degree and lies in the principal block of $N$,
 the principal block of $G$ is the only block of $G$, and yet no
 irreducible constituent of $\theta^G$ is $p$-rational. The smallest counterexample
 we have found is the {\tt SmallGroup(216,158)} for $p=3$ (see [GAP]).

\section{Proof of the main results}
In this section we prove the main results in this paper,
assuming Theorem \ref{simple} below on simple groups,
which we will prove in the next section.

% The {\em only if} part of this is a direct consequence of the main result of \cite{NTV14}. We thus want to prove that if $p$ is an odd prime, and the principal ($p$-)block of $G$ does not contain any non-trivial $p$-rational character of $p'$-degree, then $\norm G P =P\times K$, where $P\in \syl p G$. We proceed by induction and reduce this question to a question on finite simple groups. We begin this section with the results on finite simple groups that we will need later on this work.

\begin{thm}\label{simple} Let $p$ be an odd prime. Let $S \nor G$, where $\cent G S =1$ and $S$ is a non-abelian simple group of order divisible by $p$. Suppose that  $G/S$ is a $p$-group. Then $G$ has a self-normalizing Sylow $p$-subgroup if and only if there is no nontrivial $p$-rational character of $p'$-degree in the principal block of $G$.
\end{thm}

In several parts of this paper, we will use the fact that $\irr{B_0(G/N)} \sbs \irr{B_0(G)}$ if $N \nor G$. (See, for instance,
the discussion before Theorem 7.6 of \cite{Nav98}.)

\begin{cor}\label{Tiep2}
Let $p$ be an odd prime. Suppose that $G$ is a finite group
such that $G=NP$, where $P \in \syl p G$ and $N \nor G$ is a direct product of non-abelian simple groups of order divisible by $p$. If there are no non-trivial $p$-rational irreducible
characters of $p'$-degree in the principal block of $G$, then $P=\norm GP$.
\end{cor}

\begin{proof} We proceed by induction on $|G|$.
Suppose that $L\neq M$ are 
proper normal subgroups of $G$ contained in $N$ such that $L\cap M=1$ (i.e. $P$ is not transitive on the simple normal factors of $N$). By induction,
we have that $\norm G P L=PL$ and $\norm G P M=PM$. Then 
$$\norm G P \leq \norm G P L \cap \norm G P M=PL\cap PM=P(L\cap M)=P.$$
 Hence we may assume that $N$ is a minimal normal subgroup of $G$. Write 
 $$N= S_1 \times \cdots \times S_t,$$ 
 where $S_i=S_1^{u_i}$ for some $u_i \in P$.  Write $H=\norm G {S_1}$, $P_1=P\cap H$, $Q=P\cap N$ and $Q_1=Q\cap S_1$. By Lemma 4.1 and Lemma 2.1(ii) of \cite{NTT07} we have that:
$P$ is self-normalizing in $G$ if, and only if, $\cent {\norm  N Q /Q} P=1$ if, and only if, $\cent{\norm {S_1} {Q_1}/Q_1} {P_1}=1$ if,
 and only if, $P_1$
is self-normalizing
in $S_1P_1$. Hence it suffices to show that $P_1$ is self-normalizing in $S_1P_1$. 
Assume the contrary.
Let $\overline H=H/C$, where $C =\cent G {S_1}$. We have that $S_1\cong
S_1C/C=\overline{S_1} \nor \overline H$, $\overline H /\overline{S_1}$ is a $p$-group and
 $\cent{\overline H} {\overline{S_1}}=1$. 
 We have that
 $\overline{P_1}=P_1C/C \in \syl p {\overline H}$, and
 $\overline H = \overline{S_1} \, \overline{P_1}$.
 We can check that $\overline{P_1}$ 
is not self-normalizing in $\overline H$. By Theorem \ref{simple}, $\overline H$ has a non-trivial $p$-rational character $\gamma$ of $p'$-degree in the principal block.  Let $\gamma_1=\gamma_{S_1}$. Then $\gamma_1 \in \irr{S_1}$ is $P_1$-invariant and lies
in the principal block of $S_1$. By Lemma 4.1 of \cite{NTT07}, we have that 
$$\theta =\gamma_1^{u_1}\times \cdots \times \gamma_1^{u_t} \in \irr N$$ 
is $P$-invariant. Of course $\theta$ is $p$-rational of $p'$-degree and lies in the principal block of $N$. By Lemma \ref{pquotient}, we get a contradiction.
\end{proof}

The following easy observation is stated as a lemma for the reader's convenience.

\begin{lem}\label{normalizers}
Let $N$ and $M$ be distinct normal subgroups of a group $G$.
 Let $P$ be a Sylow $p$-subgroup of $G$. Suppose that $\norm{G/N}{PN/N}$ and $\norm{G/M}{PM/M}$ have
 a normal $p$-complement. If $N\cap M=1$, then $\norm G P$ 
 has a normal $p$-complement.
\end{lem}
\begin{proof}
By elementary group theory,
$\norm{G/N}{PN/N}=\norm GP N/N$.
Hence we have that 
$$\norm G P /\norm N P\cong \norm G P N /N$$ 
has  a normal $p$-complement. Similarly, 
$\norm G P /\norm M P$ 
has a normal $p$-complement. Hence also
 $$\norm G P=\norm G P/(\norm N P \cap \norm M P)$$ has a normal $p$-complement.
\end{proof}

We are now ready to prove the main result of this paper, which is Theorem D of 
the introduction (recall that this is equivalent to Theorem A by using  Corollary 6.13 of \cite{Nav98}).
\begin{thm}\label{main}
Let $p$ be an odd prime. Let $G$ be a finite group and let $P \in \syl p G$. 
Then $\norm G P$ has a normal
$p$-complement if and only if
 the only $p$-rational  irreducible
 character of $p'$-degree lying in the principal block of $G$ is the principal character 
 of $G$.
\end{thm}
\begin{proof}

Suppose that $\norm GP$ has a normal $p$-complement.
By Theorem A of \cite{NTV14}, we have that there is a canonical bijection
$$^*: {\rm Irr}_{p'}(B_0(G))  \rightarrow {\rm Irr}_{p'}(B_0(\norm GP)).$$
In fact, if $\chi  \in {\rm Irr}_{p'}(B_0(G))$,
then $\chi_{\norm GP}=\chi^* + \Delta$,
where $\chi^* \in \irr{\norm GP}$ is linear in the principal block of $\norm GP$,
and $\Delta$ is zero or a character such that all its irreducible
constituents have degree divisible by $p$. In particular, we
 see that $^*$ commutes with
 the action of the absolute Galois group $\mathcal G$,
 and therefore
$\Q(\chi)=\Q(\chi^*)$. Since $\norm GP$ has a normal $p$-complement $X$,
we have that $\chi^* \in \irr{\norm GP/X}$ is the character of an
odd-order $p$-group $P$. Hence $\chi^*$ is never $p$-rational, unless $\chi^*=1$.
Therefore, unless $\chi=1$.  This proves one direction.

\medskip

We assume now that
the only $p$-rational  irreducible
 character of $p'$-degree lying in the principal block of $G$ is the principal character 
 of $G$, and we prove that $\norm GP$ has a normal $p$-complement,
 by induction on $|G|$.

\medskip

{\sl Step 1. We may assume that $G$ has a unique minimal normal subgroup $N$.
Also $\norm GP N/N$ has a normal $p$-complement $V/N \leq K/N=\oh{p'}{G/N}$,
and $G/K$ has self-normalizing Sylow $p$-subgroups.} 
\smallskip

Let $N$ and $M$ be distinct minimal normal subgroups of $G$. Since 
$$\irr{B_0(G/N)}\sbs \irr{B_0(G)},~~~\irr{B_0(G/M)}\sbs \irr{B_0(G)},$$
by the inductive hypothesis,
 we have that $G/N$ and $G/M$ have
Sylow normalizers with a normal $p$-complement. By Lemma \ref{normalizers},   $\norm G P$ has a normal $p$-complement too.

Write $\norm GPN/N=PN/N \times V/N$.
By Theorem 3.2 of \cite{NTV14}, we have that $V/N\leq {\bf O}_{p'}(G/N)=K/N$. 
Also, $\norm GP K=PK$, and $G/K$ has self-normalizing Sylow
$p$-subgroups.
\smallskip

{\sl Step 2. We may assume that
 $N$ is not a $p'$-group. In particular ${\bf O}_{p'}(G)=1$.}
\smallskip

We know that  $\norm {G/N}{PN/N}=PN/N\times V/N$. If $N$ is a
$p'$-group, then $V$ is a normal $p$-complement of $\norm G P N$. Hence $V\cap \norm G P \nor \norm G P$ is a $p$-complement of $\norm G P$ and we are done.

\smallskip

{\sl Step 3. We may assume $N$ is not a $p$-group. In particular, $N$ is a direct product of isomorphic non-abelian simple groups of order divisible by $p$.}
\smallskip

Suppose that $N$ is a $p$-group. 
We know that 
$$\norm GP/N=\norm {G/N}{P/N}=P/N\times V/N,$$ 
so that $K$ is a $p$-solvable group and ${\bf O}_{p}(K)=N$. Recall that ${\bf O}_{p'}(K)=1$, by Step 2. By Hall-Higman Lemma 1.2.3 $\cent K N \leq N$. We have that 
$${\bf O}_{p'}(\norm G P)\leq \cent K N \leq N.$$ 
Hence ${\bf O}_{p'}(\norm G P)=1$.
By Problem (4.8) of \cite{Nav98}, we have that $G$ has a unique $p$-block
of maximal defect, namely the principal one. Consequently every irreducible character of $G$ of $p'$-degree lies in $B_0(G)$. We have that $V/K=\cent{K/N}{P/N}$. By the Glauberman correspondence 
$$|{\rm Irr}_P (K/N)|=|\irr{V/N}|.$$
If $V/N=1$, then $\norm G P =P$ and we would be done. Hence we may assume that there is some non-trivial $\gamma \in {\rm Irr}_P(K/N)$. In particular, $\gamma$ is $p$-rational and
 has $p'$-degree. Since $G/K$ has a self-normalizing Sylow $p$-subgroup, we have that Theorem 6.1 of \cite{NTT07} produces a $p'$-degree $p$-rational character $\chi \in \irr G$ lying over $\gamma$. Since $1 \neq \chi$ lies in the principal block of $G$ we get a contradiction. 

\smallskip

{\sl Step 4. We may assume that $PN\nor G$. Hence $G/N=K/N \times PN/N$}. 
\smallskip

Recall that by induction 
$$\norm {G/N}{PN/N}=PN/N\times V/N,$$ 
where $V/N\leq {\bf O}_{p'}(G/N)=K/N$.
Notice that  $V/N=\cent{K/N}{PN/N}$. Let $\gamma \in \irr{PV}$ be $p$-rational of $p'$-degree lying in $B_0(PV)$. Hence, $\gamma_V \in \irr V$ lies in $B_0(V)$. By the relative Glauberman correspondence, Theorem \ref{relativeGlauberman}, there is a unique $\tau \in {\rm Irr}_P(K)$ such that $\tau^*=\gamma_V$. Also $\tau$ lies 
in $B_0(K)$. By Corollary \ref{thecor}, there exists $\chi \in \irr G$ over $\tau$
which is $p$-rational of $p'$-degree and lies in the principal block. By assumption, $\chi$ is the trivial character and hence $\tau=1$. We conclude $\tau^*=\gamma_V=1$. Now, $\gamma \in \irr{PV/V}$ is linear and rational. Since $p$ is odd, it must be $\gamma=1_{PV}$. We have shown that $PV=\norm G PN$ has a unique $p$-rational irreducible character of $p'$-degree in its principal block. If $PV<G$, then by induction $\norm {PV} P=\norm G P$ is $p$-decomposable. Hence, we may assume $PV=G$. In particular, $PN\nor G$ and $V=K$. 

\smallskip

{\sl Step 5. Let $Q=P\cap N \in \syl p G$. We may assume $N\cent K Q =K$.} 

\smallskip

By the Frattini argument $G=N\norm G Q$. Then $N\cent K Q \nor K$.  Assume that $N\cent K Q <K$. By Lemma 3.1 of \cite{NT11}, we have that $B_0(K)$ is the unique block of $K$ that covers the principal block of $N\cent K Q$. Let $1\neq \gamma \in \irr{K/N\cent K Q}$. Then
$\gamma$ lies in $B_0(K)$ and is $p$-rational of $p'$-degree. Since $G/N=K/N\times PN/N$, by Lemma \ref{pquotient}, we have that $\gamma$ extends to a $p$-rational character of $p'$-degree in $B_0(G)$. This is a contradiction because $1\neq \gamma$.

\smallskip

{\sl Step 6. We may assume that $p=3$ and that
$N$ is a direct product of groups isomorphic to ${\rm PSL}_2(3^{3^a})$, for some $a\geq 1$.
In particular, $Q$ is abelian.
Also $NP<G$.}

\smallskip

Let $\eta \in \irr{PN}$ be $p$-rational of $p'$-degree lying in $B_0(PN)$. Then $\nu=\eta_N\in \irr N$ is $P$-invariant, $p$-rational of $p'$-degree and
lies in $B_0(N)$. By Theorem 3.2 of \cite{NT11}, $\nu$ extends to a unique $\hat \nu \in \irr{K_\nu}$ in $B_0(K_\nu)$,
where $K_\nu$ is the stabilizer of $\nu$ in $K$. In particular, 
by uniqueness, we have that
$\hat \nu$ is $p$-rational and $P$-invariant. Write $\rho=(\hat \nu)^K\in \irr K$. Then $\rho$ is $p$-rational, $P$-invariant and of $p'$-degree.  By Lemma \ref{pquotient} we conclude that $\rho$
has an extension to a $p'$-degree
$p$-rational character in the principal block of $G$.
We conclude that $\rho=1_K$. This implies $\hat \nu =1_K$. Hence $\nu=1_N$. Thus $\eta \in \irr{PN/N}$ is linear and rational. Since $p$ is odd, this implies $\eta=1$. We have proved that $PN$ has a unique $p'$-degree $p$-rational irreducible character in $B_0(PN)$.
If $G=NP$, then the theorem follows from Corollary \ref{Tiep2}.
 Hence,
 we may assume that
  $PN<G$. Then, 
  by the inductive hypothesis,
   $\norm{PN}P=P\times Y$. By Theorem 3.2 of \cite{NTV14}, we have that $Y\leq {\bf O}_{p'}(PN)\leq N$. Hence $Y=1$ (by Step 3). By the main result of \cite{GMN04}, we have that the non-abelian composition factors of $PN$ are of type ${\rm PSL}_2(3^{3^a})$ with $a\geq 1$. 

\smallskip

{\sl Step 7. The final contradiction.} 
\smallskip

We have that $\cent K Q =Y_0 \times Q$, where $Y_0$ is a $p'$-group. 
 If 
$$N=S_1\times \cdots \times S_t,$$ 
where each $S_i$ is isomorphic to ${\rm PSL}_2(3^{3^a})$, then write $Q_i=Q\cap S_i \in \syl p {S_i}$. Since $Y_0$ centralizes $1 \ne Q_i$,
 it follows that $Y_0$ normalizes $S_i$. We have that 
$$Y_i=Y_0\cent G {S_i}/\cent G {S_i}\leq {\rm Aut}(S_i)$$ 
centralizes $Q_i$. By Lemma 3.1(i) of \cite{NTV14}, it follows that $Y_i=1$. Thus $Y_0 \leq \cent G {S_i}$ for every $i$, and so 
 $Y_0 \leq \cent G N$. By Step 1,
 $G$ has a unique minimal normal subgroup,
 so $\cent G N=1$. Hence $Y_0=1$ and $K=N$. 
 This implies that $G=NP$, but this is impossible
 by Step 6.
\end{proof} 

\section{Simple Groups}
The aim of this section is to prove Theorem \ref{simple}. We begin with some observations.

\begin{lem}\label{red}
Let $p$ be a prime and let $S$ be a normal subgroup of $G$ of $p$-power index.
\begin{enumerate}[\rm(a)]
\item The principal  block $B_0(G)$  is the only  block of $G$ that covers the 
principal  block $B_0(S)$ of $S$.
\item Suppose that $p > 2$ and that $\Irr(S) \cap B_0(S)$ contains a rational
$G$-invariant character $\alpha$. Then $\alpha$ extends to a rational character 
$\beta \in B_0(G)$.  
\item Theorem \ref{simple} holds if $G$ has a self-normalizing Sylow $p$-subgroup $P$. 
\end{enumerate} 
\end{lem}

\begin{proof}
(a) By Green's Theorem 8.11 of \cite{Nav98}, $1_G$ is the unique irreducible $p$-Brauer character of $G$ that lies 
above $1_S$. Hence the statement follows.

\smallskip
(b) By \cite[Lemma 2.1]{NT1}, $\alpha$ has a unique real extension $\beta$ to $G$, whence $\beta$ is also rational.
Since $\alpha \in B_0(S)$, $\beta \in B_0(G)$ by (a).

\smallskip
(c) By \cite[Theorem A]{NTT07}, $1_G$ is the unique $p$-rational irreducible character of $p'$-degree of $G$, whence
the claim follows.
\end{proof}

By virtue of Lemma \ref{red}(c), it remains to prove the ``if'' direction of Theorem \ref{simple}.

\begin{lem}\label{spor}
Theorem \ref{simple} holds if $S$ is either $\tw2 F_4(2)'$ or one of the $26$ sporadic simple groups.
\end{lem}

\begin{proof}
Direct computation using \cite{GAP}; note that in this case $G=S$. 
\end{proof}

\begin{lem}\label{alt}
Theorem \ref{simple} holds if $S = \AAA_n$, $n \geq 5$, is an alternating group.
\end{lem}

\begin{proof}
As mentioned above, it suffices to prove the ``if'' direction of Theorem \ref{simple}, that is,
$B_0(G)$ contains a nontrivial $p$-rational irreducible character of $p'$-degree, where $G = S = \AAA_n$.
Let $H := \SSS_n$.

Suppose that $p|(n-1)$. Then the character $\chi$ of $H$ labelled by the partition $(n-2,2)$ has degree 
$n(n-3)/2 \geq 5$ that is coprime to $p$, and $p$-core $(1)$, whence $\chi \in B_0(H)$. Since
$\chi_S$ is irreducible, it has the desired properties.

Assume now that $p \nmid (n-1)$. Write $n-1 = \sum^k_{i=1}a_ip^i$ 
with $a_i \in \ZZ$, $0 \leq a_i < p$, and $a_k > 0$. Then the character $\chi$ of $H$ labelled by the partition $(n-p^k,1^{p^k})$ has degree 
$\binom{n-1}{p^k} > 1$ that is coprime to $p$, and the same $p$-core as of $1_H$, whence $\chi \in B_0(H)$. Since 
$n \neq 2p^k+1$, the partition $(n-p^k,1^{p^k})$ is not self-associate, and so $\chi_S$ is irreducible and has the desired properties.
\end{proof}

In the case of Theorem \ref{simple} where $S$ is a simple group of Lie type, we will actually prove more than what is needed for the 
``if'' direction; we believe the established results will be useful in other applications as well. We refer the reader to \cite{C} and 
\cite{DM} for basics on complex representations of finite groups of Lie type.
 
\begin{thm}   \label{cross}
 Let $S \not\cong \tw2 F_4(2)'$ be a finite simple group of Lie type in
 characteristic~$r$ and $p\ne r$ an odd prime. Then there exists a non-trivial,
 rational-valued, $\Aut(S)$-invariant, unipotent character of $p'$-degree
 that belongs to the principal $p$-block of $S$. 
\end{thm}

\begin{proof}
(i) We work in the following setting. Let $\bG$ be a simple, simply connected
linear algebraic group over an algebraic closure of $\FF_r$ with a Steinberg
map $F:\bG\rightarrow\bG$ such that $S=G/\bZ(G)$ where $G=\bG^F$. (This is
possible since $S \not\cong\tw2F_4(2)'$.) The unipotent characters of $S$ are
then (by definition) precisely the unipotent characters of $G$ (which all have
$\bZ(G)$ in their kernel). 
%We let $uc_p(G)$ denote the number of unipotent characters in the principal
%$p$-block of $G$.

\smallskip
(ii) We first assume that $F$ is a Frobenius endomorphism defining an
$\FF_q$-rational structure on $\bG$ (i.e., $G$ is not a Suzuki or Ree group).
We let $d$ denote the order of $q$ modulo~$p$. By results mainly of
Brou\'e, Malle, and Michel, and of Cabanes and Enguehard, summarised in
\cite[Theorem A]{KM16}, the unipotent characters in a block of $G$ are unions
of $d$-Harish-Chandra series. Moreover, individual $d$-Harish-Chandra series
are in bijection with irreducible characters of the corresponding relative
Weyl groups (see \cite[Theorem~B]{KM16}). Thus by the degree formula for Lusztig
induction, the blocks of maximal defect are those parametrized by cuspidal
pairs $(L,\lambda)$ with $d$-cuspidal $\lambda \in \Irr(L)$  of degree coprime to~$p$, hence
with the $d$-split Levi subgroup $L$ having a $d$-torus in its centre, so
with $L$ being the centralizer $\bC_G(\bT)$ of a Sylow $d$-torus $\bT$ of $\bG$.
In particular if $\bC_\bG(\bT)$ is a maximal torus of $\bG$, that is, 
if $d$ is a regular number (in the sense of Springer) for the Weyl group $W$ of $G$,
then there is just one
such block, which must be the principal block. In this case, the Steinberg
character lies in the principal block, is rational and $\Aut(G)$-invariant
(see e.g.\ \cite[Theorem 2.5]{MaExt}), and its degree is a power of $r$, hence
coprime to~$p$, so we are done.

\smallskip
Consider the case $G$ is of exceptional type. Then all relevant numbers are
regular for $W$ unless $G$ is of type $E_7$ (see the tables given in
\cite{BMM}). Hence we only have to consider the
latter type. The non-regular numbers are $d=4,5,8,10,12$. Here, eight unipotent
characters are irrational (those lying in the Harish-Chandra series above the
two cuspidal unipotent characters of $E_6$, those two in the principal series
belonging to the non-rational characters of the Hecke algebra, and the two
cuspidal unipotent characters). It is immediate from the explicit list of
$d$-Harish-Chandra series in \cite[Tab.~2]{BMM} that in each case there exists
a unipotent character of $p'$-degree in the principal block that is $\Aut(S)$-invariant.
(This concerns the lines~24, 30, 34, 37 in loc.~cit.)

\smallskip
Now assume that $S$, and hence $\bG$, is of classical type.
%; in particular, $p$ is good for $\bG$.
Then the unipotent characters are uniquely determined by their multiplicities
in the Deligne--Lusztig characters and hence in particular they are rational.
Moreover, all unipotent characters are invariant under all outer automorphisms
of $S$ unless either $G$ is of type $D_n$ with $n\ge4$, or $G$ is of type
$B_2$ in characteristic $r=2$, see \cite[Theorem~2.5]{MaExt}. Since the relative
Weyl group of any non-trivial $d$-torus is a non-trivial complex reflection
group, it has a non-trivial linear character $\psi$. The unipotent character
in the principal block parametrized by $\psi$ then has degree congruent
to~1 modulo~$p$ by \cite[Theorem~4.2]{MaH0} and is not the trivial character, and
hence we are done except for types $D_n$ and $B_2$. In the cases of types
$D_4$ and $B_2$, again all relevant $d$ are regular for $W$, and thus the
Steinberg character does the job. So now assume that $G$ is of type $D_n$ with
$n\ge5$. According to \cite[Theorem 2.5]{MaExt} the unipotent characters
not stable by outer automorphisms are those labelled by degenerate symbols.
On the other hand, the unipotent characters in the principal block are
those labelled by symbols with $d$-core (respectively $e$-cocore if $d=2e$ is
even) being the symbol of the trivial character (see \cite[\S3A]{BMM}). Clearly the
$d$-core (respectively $e$-cocore) of a degenerate symbol is again degenerate,
and the symbol for the trivial character is only degenerate when $n=0$. But
in this case, $n$ is divisible by $d$ (respectively by $e$) and then $d$ is a
regular number for $W$, whence we conclude as before.

\smallskip
(iii) Finally we deal with the case of Suzuki and Ree groups. The theory of
$d$-Harish-Chandra series and $p$-blocks holds with minor modifications in this
case as well, see \cite{BMM} and \cite{MaH0}. And again all numbers $d$ are
regular for the corresponding Weyl groups, whence the Steinberg character has the
desired properties.
\end{proof}

\begin{thm}   \label{defi}
Let $S$ be a finite simple group of Lie type defined over a field of characteristic $p > 2$. If $p = 3$, 
assume in addition that $S \not\cong \PSL_2(3^{2a+1})$ for any $a \in \NN$. Then $S$  has a non-trivial, 
rational-valued, $\Aut(S)$-invariant, irreducible character of $p'$-degree that belongs to the principal block of $S$. 
\end{thm}

\begin{proof}
We keep the notation $(\bG,F,G)$ as in Step (i) of the proof of Theorem \ref{cross}.
According to \cite[Theorem, p.69]{Hum}, $B_0(S) = \Irr(S) \setminus \{\St\}$, if $\St$ denotes the Steinberg character of $G$ 
(and $S$). In particular, any irreducible character of $p'$-degree of $S$ belongs to $B_0(S)$.

First we note that the result in the case $S$ is an exceptional group of Lie type, respectively $S = \PSL_n(q)$ or ${\mathrm {PSU}}_n(q)$ with 
$n \geq 3$, has already been established in Example 5.3(a), (c), Proposition 5.5, and Proposition 5.10 of \cite{NT2}, respectively.  

In the remaining cases (and viewing $\SL_2(q)$ as $\Sp_2(q)$), we have that $G = \Sp(V)$ or ${\mathrm {Spin}}(V)$ for a suitable vector 
space $V$ over $\FF_q$. Let the pair $(\bG^*,F^*)$ be dual to $(\bG,F)$, and set $G^* := (\bG^*)^{F^*}$, so that $G^* = {\mathrm {SO}}(W)$,
${\mathrm {PCSp}}(W)$, or  ${\mathrm {PCO}}(W)^0$, where $W = \FF_q^n$ for a suitable $n \in \NN$. 
If $p \neq 3$, it is easy to see that $G^*$ has a unique conjugacy class 
of rational elements $s \in [G^*,G^*]$ of order $3$ such that an inverse image in ${\mathrm {GL}}(W)$ of order $3$ of $s$ has a fixed point subspace of dimension $n-2$
on $W$. Likewise, if $p = 3$ and $n \geq 4$, then $G^*$ has a unique conjugacy class 
of rational elements $s \in [G^*,G^*]$ of order $5$ such that an inverse image in ${\mathrm {GL}}(W)$ of order $5$ of $s$ has a fixed point subspace of dimension $n-4$
on $W$. Finally, if $S = {\mathrm {PSp}}_2(3^{2a})$ (and so $G^* = {\mathrm {SO}}_3(q)$ with $q = 3^{2a} \equiv 1 (\bmod 8)$), we can choose $\gamma \in \FF_q^\times$ of
order $8$, $t = \diag(1,\gamma,\gamma^{-1}) \in G^*$, and $s = t^2 \in [G^*,G^*]$. In all cases, $s$ has connected centralizer in $\bG^*$. It follows that the 
corresponding semisimple character $\chi_s$ of $G$ is irreducible, trivial at $\bZ(G)$, rational-valued, of degree 
$|G^*:\bC_{G^*}(s)|_{p'} > 1$, and $\Aut(S)$-invariant.
\end{proof}

\begin{proof}[Proof of Theorem \ref{simple}]
By Lemma \ref{red}(c), it remains to prove the ``if'' direction of the theorem. By Lemmas \ref{spor} and \ref{alt} we may assume 
that $S \not\cong \tw2 F_4(2)'$ is a simple group of Lie type. 
If $(S,p) \neq (\PSL_2(3^{2a+1}),3)$, then Theorems \ref{cross} and \ref{defi}
yield a non-trivial $\Aut(S)$-invariant rational irreducible character of $p'$-degree in $B_0(S)$, whence the same holds for $B_0(G)$ by   
Lemma \ref{red}(b). Assume now that  $(S,p) = (\PSL_2(3^{2a+1}),3)$ and that $P \in \Syl_p(G)$ is not self-normalizing. 
By direct computation (or by using \cite[Theorem A]{NTT07}), one sees that $\Irr(G)$ contains a non-trivial $p$-rational 
irreducible character of $p'$-degree $\chi$. Since $|G/S|$ is a $p$-power, $\chi_S$ is irreducible and belongs to $B_0(S)$ 
as we noted in the proof of Theorem \ref{defi}, whence $\chi \in B_0(G)$ by Lemma \ref{red}(a). 
\end{proof}

\section{Theorem C and final remarks}

We start this section by proving Theorem C of the introduction,
which is implied by the
deepest parts of the block theory of $p$-solvable groups.
We assume that the reader is   familiar
with the theory of blocks and normal subgroups (see, for instance,
Chapter 9 of \cite{Nav98}).

Recall that
if $B$ is a block of $G$ with defect group $P$, then
$B$ uniquely determines, up to $\norm GP$-conjugacy, a
  $p$-defect zero character $\theta \in \irr{P\cent GP/P}$ lying in a block
$b$ of $P\cent GP$ that induces $B$  (see discussion after Theorem 9.12 in \cite{Nav98}).
This character $\theta$ is called a canonical character of $B$. We first need to prove the following lemma.  

\begin{lem}\label{easy}
Suppose that $B$ is a block of
$G$ with normal defect group $P$. Let $\theta \in \irr{P\cent GP /P}$
be a canonical character of $B$.  Then:
\begin{enumerate}[\rm(a)]
\item
All irreducible Brauer characters in $B$ have height zero.

\item
 $l(B)=1$ if and only if
$\theta$ is fully ramified in $G_\theta/P\cent GP$.

\item

Suppose that $p$ is odd. Then $l(B)=1$ if and only if
there is a unique $p$-rational   $\chi \in \irr B$
such that $\chi^0 \in \ibr G$.

\end{enumerate}
\end{lem}

\begin{proof}
Let $C=\cent GP \nor G$ and  $L=PC \nor G$. Let $b$ be a block
of $L$ covered by $B$. We know that $B=b^G$
by Corollary 9.21 of \cite{Nav98}, and that $b$ has defect group $P$
(for instance, by Lemma 4.13 and Theorem 4.18
of \cite{Nav98}). By Theorem 9.12 of \cite{Nav98},
we may assume that $\ibr b=\{\theta^0\}$. 
By using Theorem 9.2 of \cite{Nav98},
we  conclude that $\ibr{B}=\ibr{G|\theta^0}$.
Let $T$ be the stabilizer of the Brauer character $\theta^0$ in $G$.
Since $\theta$ vanishes off $p$-regular elements,
we also have that $T=G_\theta$. 
We even have that $T$ is the stabilizer of $b$ in $B$,
by using Theorem 9.12 of \cite{Nav98}.
Recall that $T/L$ is a
$p'$-group, by Theorem 9.22 of \cite{Nav98}.
By the Fong-Reynolds correspondence (Theorem 9.14 of \cite{Nav98})
it is enough to prove that the irreducible
Brauer characters of $T$ lying over $\theta^0$ have height
zero. This is clear, using that $T/L$
is a $p'$-group. This proves part (a).

Also, we see that
 $|\ibr{B}|=1$ if  and only if  $|\ibr{G|\theta^0}|=1$.
By the Clifford correspondence
for Brauer characters (Theorem 8.9 of \cite{Nav98}), this happens if  and only if  $|\ibr{T|\theta^0}|=1$.
  Since $T/L$ is a $p'$-group, then
  every $\psi \in \irr{T|\theta}$ has $p$-defect zero,
and it  follows that restriction to $p$-regular
elements defines a bijection $\irr{T|\theta} \rightarrow \irr{T|\theta^0}$.
We deduce that $\theta$ is fully ramified in $T$, hence proving (b).

In order to prove (c),
we claim first that 
$\irr{G|\theta}$ is exactly the set
of $p$-rational characters in $B$ that lift irreducible
Brauer characters.
We already know that $B=b^G$ is the only block
of $G$ covering $b$, so $\irr{G|\theta} \sbs \irr B$. 
Let $\chi \in \irr{G|\theta}$ and let $\psi \in \irr{T|\theta}$
be its Clifford correspondent. Since $T/L$ is a $p'$-group,
then $\psi$ has defect zero.  In particular
$\psi^0 \in \ibr T$ and $\psi$ is $p$-rational.
Hence $\chi=\psi^G$ is also $p$-rational. By the Clifford correspondence for Brauer characters (Theorem 8.9 of \cite{Nav98}),
 we have that $\chi^0=(\psi^0)^G \in \ibr G$.
Conversely, suppose that 
$\chi \in \irr B$ is a $p$-rational
character that lifts an irreducible Brauer character. By Lemma X.2.4 of \cite{F},
we have that $P \leq \ker \chi$. By Theorem 9.12 of \cite{Nav98},
it follows that $\chi$ lies over $\theta$. This proves the claim.

 We have that 
$$|\irr{G|\theta}|=|\irr{T|\theta}|=|\irr{T |\theta^0}|=|\ibr{G|\theta^0}|=
|\ibr{B|\theta^0}| $$
and the proof of the lemma  follows.
\end{proof}

\medskip
Next is Theorem C of the introduction.

\begin{thm}
Suppose that $G$ is $p$-solvable, with $p$ odd. Let $B$ be a block
with defect group $P$ and let $b$ be its Brauer first main correspondent.
Then $l(b)=1$ if and only if there is exactly one $p$-rational $\chi \in \irr B$
of height zero and such that $\chi^0 \in \ibr B$.
\end{thm}

\begin{proof} 
Let ${\rm IBr}_0(B)$ be the set of irreducible Brauer
characters of $B$ with height zero.
By Theorem 23.9 of \cite{MW}, we know that 
$|{\rm IBr}_0(B)|=|{\rm IBr}_0(b)|$.
By Lemma \ref{easy}(b),
we have that
$|{\rm IBr}_0(B)|=|{\rm IBr}(b)|$.
Hence 
$|{\rm IBr}(b)|=1$ if and only if 
$|{\rm IBr}_0(B)|=1$.

By Theorem 10.6 of \cite{Nav98},
for each $\phi \in \ibr B$ there exists a unique
$p$-rational character $\chi \in \irr G$ such that $\chi^0=\phi$.
Hence $|{\rm IBr}_0(B)|$ is the number of $p$-rational
characters in $B$ of height zero. This concludes the proof of the statement. \end{proof}

\medskip
It is interesting to speculate up to
what level the local condition $l(b)=1$ affects the representation theory
of its global Brauer correspondent $B$.
As we have proved in this section, this condition implies that
$B$ has a unique height zero $p$-rational character $\chi$
lifting an irreducible Brauer character for $p$-solvable groups,
and for blocks with a normal defect groups. It seems that
this might be also the case for blocks with abelian defect groups.
This would follow from the Alperin weight conjecture
together with a  conjecture by G. R. Robinson on the
uniqueness of $p$-rational liftings in blocks
with a unique simple module (see \cite{MNS}).

\begin{rem}
Let $p > 2$ be a prime and let $\calO$ be the (unique up to isomorphism) 
absolutely unramified complete discrete valuation ring with $\overline\FF_p$ as its residue field.
Let $G$ be any finite group and $B$ a $p$-block of $\calO G$. Suppose that $B$ is Morita 
equivalent to a $p$-block $B'$ of $\calO H$, where $H$ is a finite $p$-solvable group, and 
suppose that the Brauer correspondent $b'$ of $B'$ satisfies $l(b') = 1$. Applying Theorem C
to $B'$ and the main result of \cite{K}, see also \cite[Corollary 1.7]{KL}, we see that there is exactly one $p$-rational $\chi \in \irr B$
of height zero and such that $\chi^0 \in \ibr B$.
\end{rem}

\medskip
We have mentioned in the introduction that 
we believe that
there might be a version of Theorem A for the prime
$p=2$. We finish this paper with the following conjecture
and some remarks on it. 

\medskip

\begin{conj}\label{Conj2}
Let $G$ be a finite group. Let $P \in \syl 2 G$.
 Then $\norm GP$ has a normal $2$-complement if and only if all
odd-degree irreducible characters in the principal $2$-block of $G$ are
$\sigma$-invariant, where $\sigma$ is the Galois automorphism that fixes
$2$-power roots of unity and squares $2'$-roots of unity.
\end{conj}

\begin{rem}
We offer some evidence in support of Conjecture \ref{Conj2}, which includes all finite solvable, symmetric, and general linear or unitary groups. 

\begin{enumerate}[\rm(i)]

\item
Suppose that $G$ is solvable.
Let $L=\oh{2'} G$. Then it is well-known that
$$\irr{B_0(G)}=\irr{G/L}.$$ 
Since $\norm GP$ has
a normal 2-complement if and only if $\norm{G/L}{PL/L}$
has a normal 2-complement, we may assume that $L=1$.
We know by the main result in \cite{I0} that there
is a natural bijection 
${\rm Irr}_{2'}(G) \rightarrow {\rm Irr}_{2'}(\norm GP)$
that commutes with Galois action.
Hence it is no loss to assume that $P \nor G$.
Assume now that $G$ has a normal 2-complement.
Then $G$ is a $2$-group, and we are done in this case.
Conversely, if all the odd-degree irreducible characters
of $G$ are $\sigma$-invariant, then all
characters of $G/P$ are $\sigma$-invariant.
Then $G=P$ by Lemma 5.1 of  \cite{N2}.

\item Suppose $G = \SSS_n$. Then $P \in \Syl_2(G)$ is self-normalizing, and certainly all 
$\chi \in \Irr(G)$ are rational-valued, hence $\sigma$-invariant.

\item More generally, suppose that $G$ is any finite group with self-normalizing Sylow $2$-subgroups. 
Then a consequence of the Galois refinement of the McKay conjecture \cite{N2} implies that all odd-degree
irreducible characters of $G$ are $\sigma$-invariant. (A reduction of this statement to quasisimple groups
has been given in \cite[Theorem 5.1]{NT5} and \cite[Theorem 3.7]{Sch}.)

\item Let $G = \GL_n(q)$ with $2|q$ and $P \in \Syl_2(G)$, chosen to be the subgroup of upper unitriangular 
matrices in $G$. Then $\bfN_G(P) = P\rtimes T$, where $T$ is the subgroup of diagonal matrices in $G$. In 
particular, $\bfN_G(P)$ has a normal $2$-complement precisely when $q=2$. The degree formula for 
unipotent characters \cite[\S13.8]{C} shows that the only unipotent character of  $\GL_k(q^l)$ of odd degree is the 
principal character. Hence Lusztig's parametrization of irreducible characters of $G$ \cite{C}, \cite{DM} 
implies that $\chi \in \Irr(G)$ has odd degree precisely when it is the semisimple character $\chi_s$ labeled
by a semisimple element $s \in G$ (if we identify the dual group $G^*$ with $G$). 
Arguing as in the proof of \cite[Lemma 9.1]{NT1}, one can show that $\chi_s$ is $\sigma$-invariant exactly when
$\chi_s = \chi_{s^2}$, i.e. when $s^2$ and $s$ are $G$-conjugate. Furthermore, \cite[Theorem, p. 69]{Hum}
implies that $\chi_s$ belongs to the principal block of $G$ precisely when $\chi_s$ is trivial at $\bfZ(G)$, 
which, by \cite[Proposition 4.5]{NT4}, is equivalent to that $s \in [G,G] = \SL_n(q)$. 
Now it is straightforward to check that $s^2$ and $s$ are $G$-conjugate
for all semisimple elements $s \in \SL_n(q)$ if and only if $q=2$. Thus Conjecture \ref{Conj2} holds in this case. 

A similar argument, applied to $\GU_n(q)$ with $2|q$, shows that Conjecture \ref{Conj2} holds in 
this case as well.
%% If $q \geq 4$, consider $s = (a,a^{-1}, 1,...,1)$ with $|a|=q+1$. If $q = 2$ and $n \geq 3$, consider
% $s = (b,b,b,1,...,1)$ with $|b| = 3$. 

\item Let $G = \GL_n(q)$ with $q$ odd and $n \geq 2$. By \cite[Theorem 2.5]{GKNT},
if $\chi \in \Irr(G)$ has odd degree, then 
$$\chi = S(s_1,\lambda_1) \circ S(s_2,\lambda_2) \circ \ldots \circ S(s_m,\lambda_m)$$
in James' notation \cite{J}, where $s_i \in \FF_q^\times$ are pairwise distinct, 
$\lambda_i \vdash k_i$, $\sum^m_{i=1}k_i = n$, and 
$$[n]_2 = [k_1]_2 < [k_2]_2 < \ldots < [k_m]_2,$$
if $[a]_2$ denotes the $2$-part of any $a \in \NN$. Furthermore, results of Fong and Srinivasan 
\cite{FS} imply that such a character belongs to the principal $2$-block of $G$ only when 
all $s_i$ are $2$-elements. Note that in this case $S(s_i,\lambda_i)$ is a product of the 
rational-valued (unipotent) character $S(1,\lambda_i)$ of $\GL_{k_i}(q)$ with a linear character of $2$-power order of
$\GL_{k_i}(q)$, whence it is $\sigma$-invariant. Since $\chi$ is obtained from the character 
$$S(s_1,\lambda_1) \otimes S(s_2,\lambda_2) \otimes \ldots \otimes S(s_m,\lambda_m)$$
of the Levi subgroup
$$\GL_{k_1}(q) \times \GL_{k_2}(q) \times \ldots \times \GL_{k_m}(q)$$
by Harish-Chandra induction, it follows that $\chi$ is $\sigma$-invariant. On the other hand,
$\bfN_G(P)$ has a normal $2$-complement if $P \in \Syl_2(G)$, 
see e.g. \cite[(5.3), (5.5)]{GKNT}).

In fact, we note that \cite[Theorem E]{GKNT} implies that Conjecture \ref{Conj2} also holds for 
$\GU_n(q)$ whenever $q$ is odd.

\item Let $G = \Sp_{2n}(q)$ with $q \equiv \pm 3 (\mod 8)$. As shown in the proof of \cite[Theorem 1]{Ko},
the normalizer of $P \in \Syl_2(G)$ contains $\SL_2(3)$ as a subgroup, and so $\bfN_G(P)$ does not
have a normal $2$-complement. It is well known, see eg. \cite[\S2]{TZ}, that $G$ has a pair of 
the so-called {\it Weil characters} $\xi_n,\eta_n \in \Irr(G)$ of degree $(q^n \pm 1)/2$, such that
the restriction of $\xi_n$ to $2'$-elements of $G$ equals to the restriction of $1_G + \eta_n$ to 
$2'$-elements of $G$. In particular, they belong to the principal $2$-block of $G$, and one of them
has odd degree. Inspecting the values of 
$\xi_n$ and $\eta_n$ at a transvection $t \in G$ \cite[Lemma 2.6]{TZ}, one can check that neither $\xi_n$ nor
$\eta_n$ is $\sigma$-invariant.
\end{enumerate}

Certainly, the arguments given in (iv)--(vi) also apply to many other finite groups of Lie type. We also note that
Conjecture \ref{Conj2} has now been reduced to almost simple groups, see \cite{NV}.
\end{rem}

\noindent
{\bf Acknowledgements.}~~We thank Gunter Malle for useful conversations
on this paper, and for kindly providing us with the proof of Theorem 5.4. We also are grateful to the referee for 
helpful comments on the paper.

\end{document}